\documentclass[a4paper,10pt]{amsart}
\usepackage{gupta}
\numberwithin{equation}{section}

\begin{document}
\title{On growth and torsion of groups}
\author{Laurent Bartholdi and Floriane Pochon}
\date{June 5th, 2007}
\begin{abstract}
We give a subexponential upper bound and a superpolynomial lower bound
on the growth function of the Fabrykowski-Gupta group.

As a consequence, we answer negatively a question by Longobardi, Maj
and Rhemtulla~\cite{rhemtulla} about characterizing groups containing
no free subsemigroups on two generators.
\end{abstract}

\maketitle

\section{Introduction}
Fabrykowski and Gupta constructed in 1985 a group of intermediate word
growth, producing in this way a new example after Grigorchuk's
original construction~\cite{grigorchuk-1}.

This group appears originally in~\cite{gupta-fabrykowski1}, and is
studied further in~\cite{gupta-fabrykowski2}; some of its algebraic
properties are explained in~\cite{bartho-grigo}.  A proof of its
intermediate growth was first given in~\cite{gupta-fabrykowski1}, with
an explicit upper bound. However, a gap in the argument lead to a
second proof, in~\cite{gupta-fabrykowski2}, this time with no upper
bound.

Although that second paper's general strategy is sound, many details
are missing or incorrect, and we hope to present here the first
complete proof. We also give explicit upper and lower bounds on the
growth function.

Let us say that two functions $f,g$ satisfy the relation $f\lesssim g$
if there is a constant $A>0$ such that $f(n)\le g(An)$. We prove the
\begin{theorem}\label{growth}
  The growth of the Fabrykowski-Gupta group is intermediate. More
  precisely, if $\gamma(n)$ denote the number of elements expressible
  as a product of at most $n$ generators of the Fabrykowski-Gupta
  group, then
$$
e^{n^\frac{\log 3}{\log 6}}\lesssim \gamma(n) \lesssim e^{\frac{n(\log\log n)^2}{\log n}}.
$$
\end{theorem}

We then apply this result to a question by Longobardi, Maj and
Rhemtulla.  Let $G$ be a group with an exact sequence $1\to N\to G\to
P\to 1$, where $N$ is locally nilpotent and $P$ is periodic. Then $G$
has no free subsemigroup. Indeed, let $x,y\in G$. Then $x^n,y^n\in N$
for some $n$ large enough, so that $\langle x^n,y^n\rangle$ is
nilpotent. Hence, neither $\langle x^n,y^n\rangle$ nor $\langle
x,y\rangle$ are free as semigroups. (Note that, without loss of
generality, one may assume that $G$ is finitely generated).

In~\cite{rhemtulla}, Longobardi, Maj and Rhemtulla asked whether the
converse were true:
\begin{question}\label{qu:1}
  Let $G$ be a finitely generated group with no free subsemigroups. Is
  $G$ a periodic extension of a locally nilpotent group?
\end{question}

The answer turns out to be negative; indeed, Ol'shanskii and Storozhev
construct in~\cite{olshanskii} a semigroup identity whose free group
is not even a periodic extension of a locally soluble group.

We remark that a very simple answer can be given to
Question~\ref{qu:1}, knowing that the Fabrykowski-Gupta group has
intermediate growth:
\begin{theorem}
  The Fabrykowski-Gupta group is generated by two elements, contains
  no free subsemigroup, and is not a periodic extension of a locally
  nilpotent group.
\end{theorem}
\begin{proof}
  Consider a short exact sequence $1\to N\to G\to P \to 1$, with $P$
  periodic. Since $G$ is not periodic (as it contains the element $at$
  of infinite order), we have $N\not= 1$. Since $G$ is just infinite,
  $N$ is of finite index in $G$, and hence, $N$ is finitely generated.
  Therefore, as $G$ has intermediate growth, so does $N$.  In
  particular, $N$ is not (locally) nilpotent.
\end{proof}
This example is quite different from the Ol'shanskii-Storozhev
example: it is a concrete, residually-$3$ group which does not satisfy
any identity.

\section{Settings}\label{gf}

\subsection{The Fabrykowski-Gupta group}

Consider the cyclic group of order three
$A=\mathbb{Z}/3\mathbb{Z}=\{0,1,2\}$ with generator $a$, and the
3-regular rooted tree $\mathcal{T}_3=A^*$, with root $\emptyset$.  The
automorphism group of $A^*$ is recursively defined by
$Aut(A^*)=Aut(A^*)\wr Sym(A)$, and every automorphism decomposes via
the map
$$
\phi\colon Aut(A^*)\to Aut(A^*)\wr Sym(A) ; \: f\mapsto \pl f_0,f_1,f_2\pr \sigma
$$
where $f_i\in A^*$ and $\sigma\in Sym(A)$.  Thus, $a$ acts on
$\mathcal{T}_3$ as a cyclic permutation of the first level $A$ of the
tree. Define the automorphism $t$ recursively by $t=\pl a,1,t\pr$.
Note that both $a$ and $t$ are of order 3.  The group $G$ generated by
$a$ and $t$ is called the Fabrykowski-Gupta group. It is known to be a
just infinite group, regular branched over $G'$ (see~\cite{bartho-grigo}).

We still call $\phi$ the decomposition map of $G$
$$
\phi\colon G\hookrightarrow G\wr A ; \: g\mapsto\pl g_i\pr_{i\in A}\, \sigma.
$$

Let $Stab(n)$ be the subgroup of $G$ that stabilizes the $n^{th}$
level of $A^*$. Then $G=Stab(1)\rtimes \langle a\rangle$. Furthermore,
$Stab(1)=\langle t\rangle^{\langle a\rangle}$.  Let $t_0=t=\pl
a,1,t\pr$, $t_1=t^a=\pl t,a,1\pr$ and $t_2=t^{a^2}=\pl 1,t,a\pr$ be
the generators of $Stab(1)$. Every word $w=w(a^{\pm 1},t^{\pm 1})$
uniquely decomposes as
\begin{align}\label{word in G}
w=t_{c_1}^{\gamma_1}\, t_{c_2}^{\gamma_2}\, \cdots\, t_{c_n}^{\gamma_n} \,\tau, \text{ with } \gamma_i\in A-\{0\},\; c_i\in A , \; c_i\not= c_{i+1}, \text{ and }\tau\in A,
\end{align}
so that the decomposition map $\phi$ is defined without ambiguity on
the set $W$ of all such words.

We define a word metric on $G$ by assigning the following weights on
the generators of $G$ : $\ell(t^{\pm 1})=1$ and $\ell(a^{\pm 1})=0$.
Then the length of a word $w\in W$, decomposed as in (\ref{word in G})
is $\ell(w)=n$. That is, the length of $w$ is the number of letters
``$t^{\pm 1}$" that appear in $w$.  The induced metric on $G$ is
$$
\ell(g) =\min\{\ell(w)| w=_G g\},
$$
for every $\gamma\in G$. We then define a minimal-length normal form
$G\to W; g\mapsto w$ on $G$.

Note that $\sum_{i\in A}\ell(g_i)\le \ell(g)$ for every $g\in G$. We
will say that $g\in G$ admits \emph{length reduction} if there is a
$d$ such that
$$
\sum_{i\in A^d} \ell(g_i)< \ell(g),
$$
where the $g_i$'s are the states of $g$ on the $d^{th}$ level of the
tree (\emph{i.e.}, the components of $\phi^d(g)$).

\section{Subexponential growth of fractal groups}


A ``traditional'' way (introduced by Grigorchuk in~\cite{grigorchuk})
to prove that a fractal group $G$ has subexponential growth is to show
that every group element admits a fixed proportion of length
reduction. More explicitly,
\begin{proposition}\cite{bux}\label{propo: grigo}
  Let $G$ be a fractal group acting on a d-regular tree, with a word
  metric $\ell$. If there exist constants $0 \le \eta <1$ and $k\ge 0$
  such that, for the natural embedding $\phi\colon
  Stab(1)\hookrightarrow G^{d} : g\mapsto \pl g_1, \ldots, g_{d}\pr$,
  $$
  \sum_{i=1}^{d} \ell(g_i)\le \eta \ell(g) +k
  $$
  for every $g\in Stab(1)$, then $G$ has subexponential growth.
\end{proposition}

\subsection{Length reduction and subexponential growth}
Let $G$ be a finitely generated fractal group acting on a $d$-regular
rooted tree, and let $\ell$ be a proper seminorm on $G$. Suppose that
for every $g=\pl g_1,\ldots ,g_d\pr \sigma$ in $G$, we have
$\sum_{i=1}^d \ell(g_i)\le \ell(g)$.

Let $\mathcal{I}_n$ be the subset of $G$ of elements that have no
length reduction up to the $n^{th}$ level of the tree. It is defined
recursively by $\mathcal{I}_0=G$ and
$$
\mathcal{I}_n=\left\lbrace g\in G\;|\; \sum_{i=1}^d \ell(g_i)= \ell(g)
  \text{ and } g_i\in I_{n-1}\text { for every } 1\le i\le
  d\right\rbrace.
$$
Then, $\mathcal{I}:=\bigcap_{n\ge 0} \mathcal{I}_n$ is the set of
words that have no length reduction on any level of the tree.

\begin{proposition}\label{propo : growth}
  Let $G=\langle X\rangle$ be a group as above, with $X$ finite, and
  $X\subset \mathcal{I}$.  If there exists some $k$ such that
  $\mathcal{I}_k$ has subexponential growth, then $G$ has
  subexponential growth.

  Moreover, if $\mathcal{I}_k$ has linear growth, then the growth of
  $G$ is bounded in the following way:
  $$
  \gamma(n) \lesssim e^{n\frac{(\log\log n)^2}{log n}},
  $$
  where $\gamma(n)=\#\{g\in G | \ell(g)\le n\}$. 
\end{proposition}

\begin{remark} 
  The idea behind this result is the following: if $\mathcal{I}$ grows
  subexponentially, then, expressing any group element $g$ of length
  $n$ as a word in $\mathcal{I}^m$ for some $m$, either $m$ is much
  smaller than $n$, and thus the set of such words grows slowly; or
  $m$ is not negligible compared to $n$ and, in that case, $g$ behaves
  as in Proposition~\ref{propo: grigo}. This kind of argument was used
  (among other works) in~\cite{bartholdi1}. Anna Erschler has obtained
  in~\cite{erschler} some similar upper bounds.
\end{remark}

In order to prove Proposition~\ref{propo : growth}, we find useful to
state two lemmas.

\begin{lemma}\label{F concave}
  Let $F$ be a map such that $\log F$ is concave. Then, for every
  $n_1,\ldots ,n_k$,
  \begin{equation}\label{ineg concave}
    \prod_{i=1}^k F(n_i)\le F\left(\frac{\sum_{i=1}^k n_i}{k}\right)^k.
  \end{equation}
  In particular, if $F$ is subexponential, then there is a map $G\ge
  F$ such that $\log G$ is concave, and hence $G$ satisfies equation
  \eqref{ineg concave}.
\end{lemma}
\begin{proof}
  By hypothesis, $\sum_{i=1}^k \log F(n_i) \le k \log
  F\left(\frac{\sum_{i=1}^k n_i}{k}\right)$. Exponentiating this last
  equation, the desired inequality follows.

  Suppose now that $F$ is subexponential, that is,
  $\lim_{n\to\infty}\frac{\log F(n)}{n}=0$. Let $(\epsilon_i)_{i\ge
    1}$ be strictly decreasing to zero and $(n_i)_{i\ge 1}$ be
  strictly increasing, such that $n_1=1$ and $\frac{\log F(n)}{n}\le
  \epsilon_i$ for every $n\ge n_i$. .  Define then $\log
  G(n)=\epsilon_i n+\delta_i$ on the the interval $n_i\le n\le
  n_{i+1}$, with $\delta_1=0$ and
  $\delta_i=(\epsilon_{i-1}-\epsilon_i)n_i+\delta_{i-1}$.  Then $\log
  G\ge \log F$ is continuous and concave by definition and $\lim_{n\to
    \infty}\frac{\log G(n)}{n}=0$.
\end{proof}

\begin{lemma}\label{lemme lf}
  Consider the maps
  \begin{equation*}
    \lambda(n)=\frac{n \log\log n}{\log n}
  \end{equation*}
  and for some $d,m>0$,
  \begin{align*}
    f(n)= \frac{\log n}{n\left(\log\log n\right)^2} &+ \frac{d^m (\log n)^2}{n(\log\log n)^2} \\ &+ \frac{n-\lambda(n)}{n} \frac{\log n}{\log \left(\frac{n-\lambda(n)}{d^m}\right)} \left(\frac{\log\log\left( \frac{n-\lambda(n)}{d^m} \right)}{\log\log n}\right)^2. 
  \end{align*}
  Then, there exists an integer $N$ such that $f(n)\le 1$ for every
  $n\ge N$.
\end{lemma}
\begin{proof}
  We write
  \[f(n)= \frac1{\log\log(n)^2}\cdot\frac{\log n}{n} (1+ A \log n) +
  \frac{\log\log(n')^2}{\log\log(n)^2}\cdot A\frac{n'\log n}{n\log(n')},
  \]
  with $A=d^m$ and $n'=(n-\lambda(n))/A$. Since
  $\frac1{\log\log(n)^2}<1$ and
  $\frac{\log\log(n')^2}{\log\log(n)^2}<1$ for $n$ large enough, it
  suffices to prove the stronger inequality
  \begin{equation}\label{eq lf}
    \frac{\log n}{n} (1+ A\log n) + A \frac{n'\log n}{n\log n'}<1
  \end{equation}
  for all $n$ large enough.

  Now this amounts to
  $$
  \frac{\log n}{n} (1+A\log n) <  1- \frac{\log n}{\log n'} +\frac{\log\log n}{\log n'};
  $$
  if we multiply this last inequality by $\frac{\log n'}{\log\log n}$,
  we get
  $$
  \frac{\log n\log n'}{n\log\log n}(1+A\log n) < 1-\frac{\log(n/n')}{\log\log n}.
  $$
  Then the LHS is bounded above by $(A+1)\frac{\log(n)^3}{n\log\log
    n}$, which tends to $0$ as $n\to\infty$; and
  $\frac{\log(n/n')}{\log\log n}$ also tends to $0$ as $n\to\infty$
  because $n/n'$ tends to $A$, so the RHS tends to $1$. It follows
  that~\eqref{eq lf} holds for $n$ large enough.
\end{proof}

\begin{proof}[Proof of Proposition~\ref{propo : growth}]


  We first suppose that $\mathcal{I}_k$ has subexponential growth, for
  some $k$. What will actually be used is that $\mathcal{I}$ itself
  has subexponential growth.

  Let us write every $g\in G$ as a product $g=g_1\cdots g_{N(g)}$ with
  $g_i\in \mathcal{I}$ and where $N(g)=\min \{k\,|\,g=g_1\cdots g_k,\,
  g_i\in \mathcal{I}\}$.

  For any $\lambda\le \frac{n}{2}$, the sphere of ray $n$ in $G$ is
  the union of
  $$W_\lambda^<(n) := \{ g\,|\, \ell(g)=n, N(g)\le \lambda\}\quad \text{and} \quad W_\lambda^>(n) := \{ g\,|\, \ell(g)=n, N(g)>\lambda\}.
  $$

  Let $\mathcal{I}(n_i)$ be the sphere of ray $n_i$ in $\mathcal{I}$
  and $\delta(n_i)=\#\mathcal{I}(n_i)$. Then, for any $k\le \lambda$,
  the cardinality of $\mathcal{I}^k\,\cap\, \{g\in G\,|\,\ell(g)=n\}$
  is $\sum_{n_1+\cdots +n_k=n} \;\prod_{i=1}^k \delta(n_i)$. Hence,
  \begin{align*}
    \# \,W_\lambda^<(n) & \le \sum_{k=1}^{\lambda}\; \sum_{n_1+\cdots
      +n_k=n} \;\prod_{i=1}^k \delta(n_i).
  \end{align*}
  We may suppose that $\delta(n)$ is increasing and, by Lemma~\ref{F
    concave}, satisfies equation \eqref{ineg concave}. Hence,
  \begin{align*}
    \# \,W_\lambda^<(n) &\le \sum_{k=1}^{\lambda} \; \sum_{n_1+\cdots +n_k=n} \delta\left(\frac{n}{k}\right)^k  \\
    & \le \sum_{k=1}^{\lambda}\; \sum_{\binom{n-1}{k-1}}
    \delta\left(\frac{n}{\lambda}\right)^\lambda \le \lambda
    \binom{n-1}{\lambda-1}
    \delta\left(\frac{n}{\lambda}\right)^\lambda.
  \end{align*}
  As $\binom{n}{\lambda}\le (\frac{en}{\lambda})^{\lambda} $ by
  Stirling's formula, it follows that
  \begin{equation}\label{W<}
    \#\,W_\lambda^<(n)\le e^\lambda\, \left(\frac{n}{\lambda}\right)^{\lambda-1}\,  \delta\left(\frac{n}{\lambda}\right)^\lambda.
  \end{equation}

  On the other hand, for $n$ fixed, there is an $m$ such that
  $\mathcal{I}_m(n)=\mathcal{I}(n)$. Therefore,
  \begin{align*}
    \#\,W_\lambda^>(n)&\le [G:Stab(m)]\, \sum_{n_1+\ldots +n_{d^m}\le
      n-\lambda}\gamma(n_1)\cdots \gamma(n_{d^m}).
  \end{align*}
  If $\gamma=\lim_{n\to\infty} \gamma(n)^{1/n} $ is the growth rate of
  $G$, then there is a constant $K>0$ such that $K\gamma\ge
  \gamma(n)^{1/n}$ for every $n\ge 1$.  Hence,
  \begin{align*}
    \#\,W_\lambda^>(n)&\le [G:Stab(m)]\,  \sum_{n_1+\ldots +n_{d^m}\le n-\lambda} K^{d^m}\gamma^{n-\lambda}
  \end{align*}
  and
  \begin{align}\label{W>}
    \#\,W_\lambda^>(n)&\le p(n) \gamma^{n-\lambda}
  \end{align}
  where $p(n)= [G:Stab(m)]\, K^{d^m} \binom{n-\lambda}{d^m}$ is a
  polynomial (of degree $d^m$).

  Set $\epsilon =\frac{\lambda}{n}$. From equations \eqref{W<} and
  \eqref{W>} we get
  \begin{align*}
    \gamma &\le  \lim_{n\to\infty}
    \left(\#\,W_\lambda^>(n)+\#\,W_\lambda^<(n)\right)^{1/n} \\
    &\le \max\left\lbrace \lim_{n\to\infty} \#\,W_\lambda^>(n)^{1/n}, \lim_{n\to\infty} \#\,W_\lambda^<(n)^{1/n}\right\rbrace 
     \le \max\left\lbrace \epsilon^{-\epsilon}
       \delta(\epsilon^{-1})^\epsilon,
      \gamma^{1-\epsilon }\right\rbrace .
  \end{align*}
  As $\lim_{\epsilon\to 0}\epsilon^{-\epsilon}
  \delta(\epsilon^{-1})^\epsilon=1$, obtain in all cases $\gamma=1$.


  Suppose next that $I_k$ grows linearly for some $k$. Thus, there is
  an $m (\ge k)$ such that $\mathcal{I}_m=\mathcal{I}$.  We have to
  show that there exist constants $A,B>0$ such that
  $$
  \gamma(n)\le \exp\left(A+B \frac{n(\log\log n)^2}{\log n}\right),
  $$
  for $n$ large enough.

  Consider the subexponential map $F(n)=e^{\frac{n(\log\log n)^2}{\log
      n}}$. Then, for $n\ge c:=e^{e^2}$,
  \begin{align*}\begin{split}
      \left(\log F(n) \right)''= \frac{1}{n(\log n)^3} \big( -\log n (\log\log n)^2 &+ 2\log n\log\log n \\
      &+ 2 (\log\log n)^2-6\log\log n+2 \big)\le 0
    \end{split}\end{align*}
  so that $\log F(n)$ is concave for $n\ge c$.

  Define $A=\log\gamma(N)$, where $N$ is as in Lemma~\ref{lemme lf}.
  Consider also the constants
  \begin{align*}
    M&=(d^m+1)[G:Stab(m)] \gamma\left(c\right)^{d^m} \left(\frac{e}{d^m}\right)^{d^m}\\
    \text{ and }\quad B&=\max\left\lbrace 2+\log
      \delta\left(\frac{n}{\lambda}\right),\,\log M+(d^m-1) A+\log
      2\right\rbrace .
  \end{align*}
  Define then the map
  \begin{equation*}
    F(n)=\begin{cases} &\exp\left(A+B\frac{n\left(\log\log n\right)^2}{\log n}\right) \quad\text{ if }\quad n\ge c\\
      \mbox{}\\
      & \exp\left(A+B\frac{c\left(\log\log c\right)^2}{\log c}\right)\quad \text{ if }\quad 0\le n<c,
    \end{cases}
  \end{equation*}
  so that $\gamma(k) \le F(k)$ for every $k\le N$. For $n>N$, let us
  show by induction that $\gamma(n)\le F(n)$.

  As before, since $\mathcal{I}=\mathcal{I}_m$, we have
  \begin{align*}
    \#\,W_\lambda^>(n)&\le [G:Stab(m)]\, \sum_{n_1+\ldots +n_{d^m}\le n-\lambda}\gamma(n_1)\cdots \gamma(n_{d^m}) \\
    &\le [G:Stab(m)]\, \sum_{n_1+\ldots +n_{d^m}\le
      n-\lambda}F(n_1)\cdots F(n_{d^m}).
  \end{align*}
  Developing this last sum and thanks to Lemma~\ref{F concave}, we get
  \begin{align*}
    \#\,W_\lambda^>(n) &\le [G:Stab(m)]\; \binom{n-\lambda}{d^m}
    \;\left(F\left(c\right)^{d^m}+\sum_{k=1}^{d^m}
      F\left(c\right)^{d^m-k}F\left(\frac{n-\lambda}{k}\right)^k\right).
  \end{align*}
  Hence,
  \begin{align*}
    \#\,W_\lambda^>(n) &\le (d^m+1)\; [G:Stab(m)]\;
    F\left(c\right)^{d^m}\; \binom{n-\lambda}{d^m} \;
    F\left(\frac{n-\lambda}{d^m}\right)^{d^m}.
  \end{align*}
  Thus, as $ \binom{n-\lambda}{d^m}\le
  \left(\frac{e(n-\lambda)}{d^m}\right)^{d^m}<\left(\frac{e}{d^m}\right)^{d^m}n^{d^m}$,
  we get
  \begin{equation*}
    \#\,W_\lambda^>(n)\le M\, n^{d^m}\, F\left(\frac{n-\lambda}{d^m}\right)^{d^m}.
  \end{equation*}
  Together with \eqref{W<}, this gives
  \begin{equation*}
    \gamma(n)\le \#\,W_\lambda^<(n)+\#\,W_\lambda^>(n)\le \left(\frac{n}{\lambda}\right)^{\lambda-1} \delta\left(\frac{n}{\lambda}\right)^\lambda + M n^{d^m} F\left(\frac{n-\lambda}{d^m}\right)^{d^m}.
  \end{equation*}
  For $\lambda=\frac{n\log\log n}{\log n}$, we see that
  $(\lambda-1)\log\left(\frac{n}{\lambda}\right)+\lambda\log
  \delta\left(\frac{n}{\lambda}\right)\le A-\log 2+B\frac{n(\log\log
    n)^2}{\log n}$, and hence
  $$
  \left(\frac{n}{\lambda}\right)^{\lambda-1}
  \delta\left(\frac{n}{\lambda}\right)^\lambda\le \frac12 F(n).
  $$
  It remains to verify that
  \begin{equation*}
    M n^{d^m} F\left(\frac{n-\lambda}{d^m}\right)^{d^m}\le \frac12 F(n).
  \end{equation*}
  But this is equivalent to
  \begin{align}\label{ineq}
    \frac{(\log M+(d^m-1) A+\log 2) \log n}{Bn(\log\log n)^2} &+
    \frac{d^m (\log n)^2}{Bn(\log\log n)^2} \nonumber\\ &+
    \frac{n-\lambda}{n} \frac{\log n}{\log
      \left(\frac{n-\lambda}{d^m}\right)} \left(\frac{\log\log\left(
          \frac{n-\lambda}{d^m} \right)}{\log\log n}\right)^2 \le 1.
  \end{align}
  As the left side of \eqref{ineq} is smaller than $f(n)$ by
  definition of $B$, this holds by Lemma~\ref{lemme lf}.
\end{proof}

\section{Growth of the Fabrykowski-Gupta group}
In the remainder, $G$ will denote the Fabrykowski-Gupta group, as
defined in Section~\ref{gf}.

\subsection{Proof of Theorem~\ref{growth}}
The lower bound is easily computed. Indeed, consider the morphism
$\psi\colon G'\to G'$ induced by $a\mapsto t$ and $t\mapsto t^a$,
where $G'=\left\langle t^{\pm a^i}t^{\mp a^j},\, i\not=
  j\right\rangle$. Since $\psi\left(t^{\pm a^i}t^{\mp
    a^j}\right)=\left\langle\left\langle t^{\pm a^i}t^{\mp
      a^j},1,1\right\rangle\right\rangle$, there is an injective map
$$
\left( B_G(n)\cap G'\right)^3\hookrightarrow B_G(6n)\cap G' \;;\; (g_1, g_2,g_3) \mapsto \psi(g_1)\psi(g_2)^a\psi(g_3)^{a^{2}}
$$
where $B_G(n)$ is the ball of radius $n$ in $G$. Hence, $\beta(6n) \ge \beta(n)^3$, with $\beta(n)=\#\,(B_G(n)\cap G')$. Iterating this inequality, one get $\beta(2\cdot 6^n)\ge \beta(2)^{3^n}=12^{3^n}$, so that
$$
\gamma(t)\ge \beta(t)\ge 12^{(t/2)^\frac{\log 3}{\log 6}}.
$$

On the other hand, the upper bound follows directly from the following
result and Proposition~\ref{propo : growth}.

\begin{proposition}\label{structure I}
  \begin{enumerate}
  \item If $w\notin \mathcal{I}$, then $w$ has length reduction up to
    the third level. Equivalently,
    $\mathcal{I}=\mathcal{I}_3$;\label{reduction}
  \item The growth of $\mathcal{I}$ is linear.\label{I linear} 
  \end{enumerate}
\end{proposition}

\noindent Before we prove Proposition~\ref{structure I}, let us give some
definitions and lemmas.

\subsection{Length reduction of words}


Consider the subsets of $A^*$
$$
\mathcal{S}=\{s| s \text{ is a subword of } (\ldots 0\,2\,1\,0\,1\,2\,0\ldots )^\sigma, \text{ for } \sigma \in A\}
$$
and
$$
\partial\mathcal{S}=\{s| s \text{ is a subword of } (\ldots 1\,1\,1\,2\,2\,2\ldots )\}.
$$
Note that 
\begin{equation}\label{equiv suites}
s=(s_i)_{i=1}^n \in A\partial \mathcal{S} \text{ if and only if } \Sigma s:=\left(-\sum_{k=1}^i s_k\right)_{i=1}^n\in \mathcal{S}.   
\end{equation}
 
For sequences $c=(c_i)_{i=1}^n \in \partial \mathcal{S}$ and $\gamma=(\gamma_i)_{i=1}^n \in S$, consider the maps
\begin{align*}
m(c)&=\begin{cases} & 1 \quad\text{ if } c \text{ is a subword of } (012)^\infty\\
& k \quad\text{ if } c_{k-1}=c_{k+1}\\
& n \quad\text{ if } c \text{ is a subword of } (021)^\infty
     \end{cases}\\
\text{ and }\\
\partial m(\gamma)&=\begin{cases} & 1 \quad\text{ if } \gamma \text{ is a subword of } 2^\infty\\
& k \quad\text{ if } \gamma_{k}=1 \text{ and } \gamma_{k+1}=2\\
& n \quad\text{ if } \gamma \text{ is a subword of } 1^\infty
     \end{cases}
\end{align*}
so that, obviously, $\partial m(\gamma)=m(\Sigma \gamma)$.

Define next, for an element $g$ written as in \eqref{word in G}, the
\emph{exponent sequence} $\gamma(g)=(\gamma_i)_{i=1}^n$ and the
\emph{index sequence} $c(g)=(c_i)_{i=1}^n$ of $g$.

As $t_{\sigma(i)}=t_i^{\sigma}$ for any $i, \sigma\in A$, the following remark holds.
\begin{lemma}\label{lemme permut}
  Let $s= t_{c_1}^{\gamma_1}\cdots t_{c_n}^{\gamma_n}\tau=\pl
  s_0,s_1,s_2\pr \tau$ be any word and its first level decomposition,
  and let $\sigma\in A$. Then
$$
s_\sigma := t_{\sigma(c_1)}^{\gamma_1}\cdots t_{\sigma(c_n)}^{\gamma_n}\tau=\pl s_0,s_1,s_2\pr^\sigma \tau.
$$ 
In particular, $s$ and $s_\sigma$ have the same first level decompositions up to a permutation of the components.
\end{lemma}

Recall that $\mathcal{I}_1=\{g\in G| \sum_{i\in
  A}\ell(g_i)=\ell(g)\}$. It is characterized in the following way.

\begin{lemma}\label{mot sans red} The set $\mathcal{I}_1$ is exactly
  the set of elements $g$ that may be written as
\begin{equation*}\label{mot sans 1-red}
g= t_{c_1}^{\gamma_1}\, \cdots\, t_{c_{m-1}}^{\gamma_{m-1}}\, t_{c_{m}}^{\gamma_{m}} \, t_{c_{m+1}}^{\gamma_{m+1}}\, \cdots\, t_{c_n}^{\gamma_n}\, \tau
\end{equation*}
with $\gamma_i=\pm 1$, and $c_i\in A$ such that 
\begin{itemize}
\item[(a)] $c(g)\in \mathcal{S}$ (with, say, $m(c(g))=m$), 
\item[(b)] If $2<m<n-2$, then $\gamma_{m-1}=\gamma_{m+1}$. 
\end{itemize}
\end{lemma}

\begin{proof}
 Suppose that $g$ satisfies (a) and (b). If $m=1$ or $m=n$, then $g$ is obviously in $\mathcal{I}_1$. Otherwise, $g$ contains a subword
\begin{align*}
s=t_{\sigma(1)}^\alpha t_{\sigma(0)}^\beta t_{\sigma(1)}^\gamma = &\pl t^{\alpha},a^{\alpha},1\pr^\sigma \pl a^{\beta},1,t^{\beta}\pr^\sigma
\pl t^{\gamma},a^{\gamma},1\pr^\sigma \\
=& \pl t^{\alpha}a^{\beta}t^{\gamma}, a^{\alpha}a^{\gamma}, t^{\beta}\pr^\sigma,
\end{align*}
where $\sigma\in A$.  
If $2<m<n-2$ and $\alpha =\gamma$ then $s=\pl t^{\alpha}a^{\beta}t^{\alpha}, a^{2\alpha}, t^{\beta}\pr^\sigma$ so that $\sum_{i\in A}\ell(g_i)=\ell(g)$.

Reciprocally, suppose that $g$ does not satisfy (a). Then $g$ contains a subword 
\begin{align*}
u&=t_{\sigma(0)}^{\alpha}t_{\sigma(1)}^{\beta}t_{\sigma(0)}^{\gamma}
=\pl a^{\alpha},1,t^{\alpha}\pr^\sigma\pl t^{\beta},a^{\beta},1\pr^\sigma \pl a^{\gamma},1,t^{\gamma}\pr^\sigma \\
&=\pl a^{\alpha}t^{\beta}a^{\gamma},a^{\beta},t^{\alpha-\gamma}\pr^\sigma ,
\end{align*}
We see that $\sum_{i\in A}\ell(s_i)\le\ell(s)-1<\ell(s)$, hence $\sum_{i\in A}\ell(g_i)<\ell(g)$.

Finally, suppose that $g$ does not satisfy (b), that is, $2<m<n-2$ and $g$ contains a subword
\begin{align*}
s=t_{\sigma(1)}^\alpha t_{\sigma(0)}^\beta t_{\sigma(1)}^{-\alpha} = \pl t^{\alpha}a^{\beta}t^{\alpha}, 1, t^{\beta}\pr^\sigma.
\end{align*}
Again, $\sum_{i\in A}\ell(g_i)<\ell(g)$.
\end{proof}

Let $g = t_{c_1}^{\gamma_1}\, t_{c_2}^{\gamma_2}\, \cdots\, t_{c_{m-1}}^{\gamma_{m-1}}\, t_{c_m}^{\gamma_m}\, t_{c_{m+1}}^{\gamma_{m+1}}\, \cdots\, t_{c_n}^{\gamma_n}\, \tau$ be a group element with $c(g)\in \mathcal{S}$ and $m(c(g))=m$. Developing $g$ on the first level, we get 
\begin{align*}
\pl g_0, g_1,g_2\pr = \, \cdots \,\pl 1,t^{\gamma_{m-2}}, &a^{\gamma_{m-2}}\pr^{a^{c_m}}\, \pl t^{\gamma_{m-1}},a^{\gamma_{m-1}},1\pr^{a^{c_{m}}} \, \pl a^{\gamma_m},1,t^{\gamma_m}\pr^{a^{c_m}} \\
 & \pl t^{\gamma_{m+1}},a^{\gamma_{m+1}},1\pr^{a^{c_m}} \pl 1,t^{\gamma_{m+2}},a^{\gamma_{m+2}}\pr^{a^{c_m}}\cdots.
\end{align*}
By Lemma~\ref{lemme permut}, up to a permutation of the components, we may suppose $c_m=0$. We may also suppose that $c_1=1$, as the two remaining cases behave symmetrically. Hence we get
\begin{align*}
g_{0}&= t^{\gamma_1}\,a^{\gamma_2}\,t^{\gamma_4}\,a^{\gamma_5}\cdots a^{\gamma_{m-3}}\, t^{\gamma_{m-1} }\, a^{\gamma_m}\, t^{\gamma_{m+1} }\, a^{\gamma_{m+3} }\, t^{\gamma_{m+4}}\cdots \\
&= t_0^{\gamma_1}\,t_{-\gamma_2}^{\gamma_4}\cdots   t_{*}^{\gamma_{m-4} }\, t_{*-\gamma_{m-3}}^{\gamma_{m-1} }\, t_{*-\gamma_{m-3}-\gamma_m}^{\gamma_{m+1} } \,t_{*-\gamma_{m-3}-\gamma_m-\gamma_{m+3}}^{\gamma_{m+4}} \cdots \\
g_{1}&=a^{\gamma_1}\,t^{\gamma_3}\,a^{\gamma_4}\,t^{\gamma_6}\cdots a^{\gamma_{m-4}}\, t^{\gamma_{m-2} } \,a^{\gamma_{m-1}+\gamma_{m+1}}\, t^{\gamma_{m+2} } \,a^{\gamma_{m+4} } \,t^{\gamma_{m+5}}\cdots \\
&= t_{-\gamma_1}^{\gamma_3}\, t_{-\gamma_1-\gamma_4}^{\gamma_6}\cdots  t_{*}^{\gamma_{m-5} } \,t_{*-\gamma_{m-4}}^{\gamma_{m-2} }\, t_{*-\gamma_{m-4}-(\gamma_{m-1}+\gamma_{m+1})}^{\gamma_{m+2}} \\ 
& \hspace{6cm} t_{*-\gamma_{m-4}-(\gamma_{m-1}+\gamma_{m+1})-\gamma_{m+4}}^{\gamma_{m+5}}\cdots \\
g_{2}&=t^{\gamma_2}\,a^{\gamma_3}\,t^{\gamma_5}\,a^{\gamma_6} \cdots a^{\gamma_{m-5}}\, t^{\gamma_{m-3} }\, a^{\gamma_{m-2}}\, t^{\gamma_{m} }\, a^{\gamma_{m+2} }\, t^{\gamma_{m+3}}\cdots \\
&= t_0^{\gamma_2}\,t_{-\gamma_3}^{\gamma_5}\cdots   t_{*}^{\gamma_{m-6} }\, t_{*-\gamma_{m-5}}^{\gamma_{m-3} }\, t_{*-\gamma_{m-5}-\gamma_{m-2}}^{\gamma_{m} } \,t_{*-\gamma_{m-5}-\gamma_{m-2}-\gamma_{m+2}}^{\gamma_{m+3}} \cdots. 
\end{align*}

Set $\tilde{\gamma}(g_0)=(\gamma_1,\;\gamma_4,\ldots,\gamma_{m-4},\;\gamma_{m-1}+\gamma_{m+1},\;\gamma_{m+4}, \ldots)$. If $\tilde{\gamma}(g_0), \gamma(g_1),\gamma(g_2)\in A\partial \mathcal{S}$, then, thanks to \eqref{equiv suites}, the following relations hold
\begin{align}
\partial m(\gamma(g_1))+1&= m(c(g_2)),\label{rel 1}\\
\partial m(\gamma(g_2)) +1&=m(c(g_0)),\label{rel 2}\\
 \partial m(\tilde{\gamma}(g_0))&=m(c(g_1)).\label{rel 3}
\end{align}

\begin{lemma} \label{lemma:cara I}
\begin{enumerate}
\item $g\in \mathcal{I}_n$ if and only if $g\in \mathcal{I}_{n-1} $ and $g_x\in \mathcal{I}_1$ for every $g\in A^{n-1}$;\label{lemme 0}
\item For every $x\in A^*$,
$$
\forall  i\in A \colon c(g_{xi})\in \mathcal{S}  \text{ if and only if }\forall i\in A^\times \colon \gamma(g_{xi})\in A\partial \mathcal{S} \text{ and }\tilde{\gamma}(g_{x0})\in A\partial \mathcal{S};
$$ \label{lemme 2}
\item $g\in \mathcal{I}\text{ if and only if }c(g_x)\in \mathcal{S} \text{ for every } x\in A^*$; \label{lemme 1}
\item  $g\in \mathcal{I}$ if and only if $\gamma(g_{xi})\in A\partial \mathcal{S}$ for $i\in A^\times$ and $\tilde{\gamma}(g_{x0})\in A\partial \mathcal{S}$ for every $x\in A^*$. \label{lemme 3}
\end{enumerate}
\end{lemma}
\begin{proof}
\begin{enumerate}
\item This follows from the definition.
\item Applying \eqref{equiv suites} to the exponent sequences of the components of $g_x$, equations \eqref{rel 1}-\eqref{rel 3} show that $\tilde{\gamma}(g_{x0})\in A\partial \mathcal{S}$ if and only if $c(g_{x1})\in \mathcal{S}$, that $\gamma(g_{x1})\in A\partial \mathcal{S}$ if and only if $c(g_{x2})\in \mathcal{S}$ and that $\gamma(g_{x2})\in A\partial \mathcal{S}$ if and only if $c(g_{x0})\in \mathcal{S}$.
\item If there exists $x\in A^*$ such that $c(g_x)\notin \mathcal{S}$, then $g_x\notin \mathcal{I}_1$ by Lemma~\ref{mot sans red}, so that $g\notin \mathcal{I}$.
Reciprocally, fix $x\in A^*$ and write $g_x= t_{c_1}^{\gamma_1}\,  \cdots\,  t_{c_m}^{\gamma_m}\,  \cdots\, t_{c_n}^{\gamma_n}\, \tau$, with $m(c(g_x))=m$. By hypothesis, $c(g_x)\in S$, so that (by \eqref{lemme 0}) it is enough to see that, if $2<m<n-2$, then $\gamma_{m-1}=\gamma_{m+1}$. But $c(g_{x1})\in \mathcal{S}$ so that $\gamma(g_{x0})\in A\partial \mathcal{S}$, by \eqref{lemme 2}. Therefore, $\gamma_{m-1}=\gamma_{m+1}$. 
\item This follows from \eqref{lemme 2} and \eqref{lemme 1}.\qedhere
\end{enumerate}
\end{proof}

\begin{lemma}\label{lemma:words not in I} 
Let $g= t_{c_1}^{\gamma_1}\, \cdots\, t_{c_m}^{\gamma_m}\, \cdots\, t_{c_n}^{\gamma_n}\tau $ be an element of $I_1$ of length $n$, with $m(c(g))=m$, and such that $\gamma(g)\in \partial \mathcal{S}$. Suppose moreover that $10<m<n-10$.
Then $g\notin \mathcal{I}$.
\end{lemma}
\begin{proof}
If $\gamma_{m-1}\not=\gamma_{m+1}$, then $g\notin \mathcal{I}_1$ by Lemma~\ref{mot sans red}. Also, if $\tilde{\gamma_0}(g)\notin \partial \mathcal{S}$, then $c(g_{1})\notin \mathcal{S}$, so $g_{1}\notin \mathcal{I}_1$. Suppose now that $ \gamma_{m-1}=\gamma_{m+1}$ and $\tilde{\gamma_0}(g)\in \partial \mathcal{S}$.  As $\gamma_0(g)\in\partial \mathcal{S}$ by hypothesis, we have $\partial m(\gamma(g_0))\in \{\frac{m-2}{3},\frac{m+4}{3}\}$. Thus, there are 6 remaining choices for $\partial m(\gamma(g))$:
\begin{itemize}
\item $\partial m(\gamma(g))=m+1$ or $\partial m(\gamma(g))=m+2$. In those cases, $\gamma_{m-4}=1=\gamma_{m+1}$ and $\gamma_{m+4}=2$. But $\partial m(\gamma(g_2))=\frac{m+1}{3}$, so that $m(c(g_0))=\frac{m+4}{3}$. 
\item $\partial m(\gamma(g))=m+3$. In that case, $\gamma_{m+1}=1$ and $\gamma_{m+4}=2=\gamma_{m+7}$. But $\partial m(\gamma(g_2))=\frac{m+4}{3}$, so that $m(c(g_0))=\frac{m+7}{3}$. 
\item $\partial m(\gamma(g))=m-2$ or $\partial m(\gamma (g))=m-3$. In those cases, we have $\gamma_{m-4}=1$ and $\gamma_{m-1}=2=\gamma_{m+1}$. But $\partial m(\gamma(g_2))=\frac{m-2}{3}$, so that $m(c(g_0))=\frac{m+1}{3}$. 
\item $\partial m(\gamma(g))=m-4$. In that case, $\gamma_{m-7}=1=\gamma_{m-4}$ and $\gamma_{m-1}=2$. But $\partial m(\gamma(g_2))=\frac{m-5}{3}$, so that $m(c(g_0))=\frac{m-2}{3}$.
\end{itemize}
In any of those cases, using Lemma~\ref{mot sans red}, we see that $g_0\notin \mathcal{I}_1$, so that $g$ does not belong to $\mathcal{I}$.
\end{proof}

\subsection{Proof of Proposition~\ref{structure I}}

\begin{enumerate}
\item \label{I=I3} Let $g= t_{c_1}^{\gamma_1}\, \cdots\, t_{c_m}^{\gamma_m}\, \cdots\, t_{c_n}^{\gamma_n}\tau \in \mathcal{I}_3$, with $m(c(g))=m$. 
For every $i\in A$, we know by hypothesis that $\gamma(g_i),\gamma(g_{ij})\in A\partial \mathcal{S}$ for $j\not=0$ and $\tilde{\gamma}(g_0),\tilde{\gamma}(g_{i0})\in A\partial \mathcal{S}$.

By Lemma~\ref{lemma:cara I} \eqref{lemme 3}, all we have to show is that $\gamma(g_{xi})\in A\partial \mathcal{S}  \text{ for } i\in A^\times \text{ and }\tilde{\gamma}(g_{x0})\in A\partial \mathcal{S}, \text{ for every } x\in  A^2 A^*$. Now,

\begin{itemize}
\item[-] For $j\in A^\times$, as $\gamma(g_{0j})\in A\partial \mathcal{S}$, the index sequence $\gamma(g_0)$ is of one of the following types

\vspace{.2cm}
\begin{center}
\begin{tabular}{cccccccc}
$\ldots$ &$\gamma_{m-7}$ & $\gamma_{m-4}$ & $\gamma_{m-1}$ & $\gamma_{m+1}$ & $\gamma_{m+4}$ & $\gamma_{m+7}$ &$\ldots$ \\[.2cm]
$\ldots$ &1&2&1&1&2&2&$\ldots$ \\
$\ldots$ &1&1&1&1&2&2& $\ldots$ \\
$\ldots$ &1&1&2&2&2&2& $\ldots$ \\
$\ldots$ &1&1&2&2&1&2& \ldots
  \end{tabular}\end{center}

\vspace{.2cm}

which means that 
\begin{equation}
\partial m\left(\tilde{\gamma}(g_0)\right) \in \left\lbrace  \frac{m-5}{3},\;\frac{m-2}{3},\;\frac{m+1}{3},\; \frac{m+4}{3}\right\rbrace.\label{m0}
\end{equation}
In any of those cases, note that we also have $\gamma(g_{00})\in A\partial\mathcal{S}$.
Altogether, this implies that 
$
\gamma(g_{0y})\in A\partial \mathcal{S}
$

for every $y\in A^*$. Hence, $\gamma(g_{0xi})\in A\partial \mathcal{S}$ for $i\in A^\times$ and $\tilde{\gamma}(g_{0x0})\in A\partial \mathcal{S}$ for every $x\in A^*$. 

\item[-] For $i\in A^\times$, since $\gamma(g_i)\in A\partial \mathcal{S}$, we have $\gamma(g_{iy})\in A\partial \mathcal{S}$ for every $y\in A^*$. Hence,
$\gamma(g_{ixj})\in A\partial \mathcal{S}$ for $j\in A^\times$ and $\tilde{\gamma}(g_{ix0})\in A\partial \mathcal{S}$ for every $x\in A^*$. 

Moreover, $\gamma(g_{ij})\in A\partial \mathcal{S}$ for $i\in A^\times$ and $j\in A$ implies that 
\begin{align}
m(c(g_1))&\in\{\partial m(\gamma(g_1))- 1,\; \partial m(\gamma(g_1))\pm 2,\; \partial m(\gamma(g_1))\pm 3, \partial m(\gamma(g_1))+4\},\label{m1}\\
m(c(g_2))&\in\{\partial m(\gamma(g_2))-1,\; \partial m(\gamma(g_2))\pm 2,\; \partial m(\gamma(g_2))\pm 3,\partial m(\gamma(g_1))+4\}.\label{m2}
\end{align}
\end{itemize}

Using relations \eqref{rel 1}-\eqref{rel 3} and \eqref{m0},\eqref{m1} and \eqref{m2}, we see that, given one of $m(c(g))$, $m(c(g_0))$, $m(c(g_1))$ or $m(c(g_2))$, the number of possibilities of choosing the three others (so that $g$ remains in $\mathcal{I}$) is bounded by a constant (independently of the length of $g$). 

\item We have to show that $\delta(n)=\#\,\mathcal{I}(n)$ is bounded by a constant (independent of $n$). But
$$
\delta(n)\le K\; \#\{\text{ possible choices of } m(c(g)) \} \;\#\{\text{ possible choices of } m(c(g_i)), i\in A \}.
$$
Let $g\in \mathcal{I}$. For $i\in A^\times$, we know by Lemma~\ref{lemma:cara I} \eqref{lemme 3} that $\gamma(g_{i}) \in A\partial\mathcal{S}$. 
Hence, by Lemma~\ref{lemma:words not in I}, we have $m(c(g_i))\le 10$ or $m(c(g_i))\ge n-10$. 
Therefore, there is at most 20 choices for $m(c(g_i))$ (to be chosen between 1 and $n$). Now, the last assertion in the proof of \eqref{I=I3} insures that the remaining choices of $m$ and $m(c(g_0))$ are bounded by a constant. \qed
\end{enumerate}

\bibliographystyle{alpha}
\bibliography{gupta}

\end{document}